\documentclass[11pt,reqno]{amsart}
\usepackage{fullpage,a4wide}

\usepackage{graphicx}
\usepackage[draft]{hyperref}
\usepackage{amsmath,amsopn,amssymb,amsfonts,stmaryrd,setspace}
\usepackage{verbatim}
\usepackage{amsthm}
\usepackage{mathtools}
\usepackage{color}
\usepackage{enumitem}
\usepackage[framemethod=TikZ]{mdframed}
\usepackage{bbm}
\usepackage{mathrsfs}
\usepackage{booktabs}
\usepackage{caption}
\usepackage{bm}
\usepackage{tensor}
\usepackage{ytableau}
\usepackage{centernot}

\newcommand{\lrb}[1]{\left(#1\right)}

\title{Eisenstein--type series associated to partition ranks}
\author[K. Bringmann]{Kathrin Bringmann}
\author[B. Pandey]{Badri Vishal Pandey}
\address{Department of Mathematics and Computer Science\\Division of Mathematics\\University of Cologne\\ Weyertal 86-90 \\ 50931 Cologne \\Germany}
\email{kbringma@math.uni-koeln.de}
\email{bpandey@uni-koeln.de, badrivishal9451@gmail.com}

\author{Jan-Willem M. van Ittersum}
\address{Department of Mathematics and Computer Science, University of Cologne,
	Weyertal 86-90, 50931 Cologne, Germany}
\curraddr{Korteweg--de Vries Institute for Mathematics, University of Amsterdam, Postbus 94248, 1090 GE  Amsterdam, The Netherlands}
\email{j.w.m.vanittersum@uva.nl}

\DeclareMathOperator{\sgn}{sgn}

\newcommand{\N}{\mathbb{N}}
\newcommand{\Z}{\mathbb{Z}}

\renewcommand{\c}{\mathbb{C}}

\newcommand{\h}{\mathbb{H}}

\renewcommand{\t}{\tau}
\newcommand{\z}{\zeta}

\theoremstyle{plain}
\newtheorem{thm}{Theorem}[section]
\newtheorem{lem}[thm]{Lemma}
\newtheorem{cor}[thm]{Corollary}
\newtheorem{prop}[thm]{Proposition}

\newtheorem{rem}[thm]{Remark}
\newtheorem*{rem*}{Remark}

\theoremstyle{definition}

\newcommand{\Pmod}[1]{\ \, ( \mathrm{mod} \, #1 )}

\numberwithin{equation}{section}

\renewcommand{\pmod}[1]{\ \left( \mathrm{mod} \, #1 \right)}
\renewcommand{\=}{\: =\: }

\allowdisplaybreaks

\makeatletter
\@namedef{subjclassname@2020}{%
	\textup{2020} Mathematics Subject Classification}
\makeatother


\begin{document}
\begin{abstract} 
In this paper, we introduce a class of functions that behave like classical Eisenstein series in many ways, but with a key distinction: only their non-holomorphic completions transform like (quasi)modular forms. We show how the partition rank generating function can be expressed in terms of partition traces of these functions. A key feature of our construction is that the completions satisfy a holomorphic anomaly equation--a phenomenon typically seen in the context of quantum field theory and string theory. We also show that the Fourier coefficients of these Eisenstein-type series are integral.
\end{abstract}
\maketitle

\section{Introduction and statement of results}
In this paper, we introduce a new class of Eisenstein-type series.\hspace{.5em}These series share several structural features with classical Eisenstein series, but differ in a crucial respect: only their non-holomorphic completions transform as (quasi)modular forms. To motivate their definition, we first explore their connection to the rank generating function, before examining their further properties in detail.

\subsection{Background}
A {\it partition} of~$n\in\N_0$ is a weakly decreasing sequence of positive integers that sum to $n$. We denote by $p(n)$ the number of partitions of $n$.
Recall the famous Ramanujan congruences
\begin{equation*}
	p(5n+4) \equiv 0\Pmod5,\quad p(7n+5) \equiv 0\Pmod7,\quad p(11n+6) \equiv 0\Pmod{11}.
\end{equation*}
To explain the first two, Dyson \cite{Dys} introduced the {\it rank} of a partition $\lambda$, which is defined as
\[
	\mathrm{rank}(\lambda) := \text{largest part of }\lambda - \text{number of parts of }\lambda.
\]
Dyson conjectured that reducing the rank (mod $5$) (resp.\ $7$) divides the partitions of $5n+4$ (resp.\ $7n+5$) into $5$ (resp.\ $7$) sets of equal size. This conjecture was proven by Atkin and Swinnerton-Dyer \cite{AS}. In the same paper, Dyson also conjectured the existence of another statistic, which he called the ``crank'' and which should explain all three partition congruences. Garvan \cite{Ga} found a crank for vector partitions and Andrews--Garvan \cite{AG88} defined a crank for ordinary partitions. Letting $o(\lambda)$ denote the number of ones in a partition~$\lambda$, and $\mu(\lambda)$ the number of parts strictly larger than $o(\lambda)$, the {\it crank} is defined as
\[
\mathrm{crank}(\lambda) :=
\begin{cases}
	\text{largest part of }\lambda & \text{if }o(\lambda)=0,\\
	\mu(\lambda)-o(\lambda) & \text{if }o(\lambda)>0.
\end{cases}
\]

Let $N(m,n)$ denote the number of partitions of $n$ with rank $m$. Its generating function is
\begin{equation}\label{eq:rank}	
R(\zeta;q):= 
 \sum_{\substack{n\ge0\\m\in\Z}} N(m,n) \z^m q^n = \sum_{n\ge0} \frac{q^{n^2}}{(\z q)_n\lrb{\z^{-1}q}_n}, 
\end{equation}
(see \cite{AS}) where $(a)_n=(a;q)_n:=\prod_{j=0}^{n-1} (1-aq^j)$ for $n\in\N_0\cup\{\infty\}$. Let $M(m,n)$ denote the number of partitions of $n$ with crank $m$, except for\footnote{The correct combinatorial values for the anomalous case $n=1$ are $1$ if $m=0$ and $0$ for $m\neq 0$.} $n = 1$ where $M(-1,1) = -M(0,1) = M(1,1) := 1$ as given by the following generating function \cite{AG88}
\begin{equation}\label{eq:crank}
C(\zeta;q):= \sum_{\substack{n\geq 0\\m\in \mathbb{Z}}} M(m,n) \zeta^m q^n = \frac{(q)_\infty}{(\z q)_\infty\lrb{\z^{-1}q}_\infty}.
\end{equation}
Modularity properties of the rank and of the crank generating function differ significantly: the crank generating function is basically a meromorphic Jacobi form, whereas the rank generating function is a ``mock Jacobi form'' \cite{BGM}: it only transforms like a Jacobi form after adding a non-holomorphic term (see Subsection~\ref{S:Completion} for the precise transformations of the rank generating function). We next consider the {\it crank moments} \cite{AG}
\[
	C_{k}(q) := \sum_{\substack{n\geq 0\\m\in \Z}} m^{k} M(m,n) q^n.
\]
These moments can be written in terms of quasimodular forms. In \cite{Traces}, the authors expressed these as a so-called partition Eisenstein trace. Here, for a sequence of functions $h=\{h_k\}_{k\in\N}$, define, for $n\in\N$, the {\it$n$-th partition trace} with respect to $h$ and a function~$\phi$ on partitions as 
\begin{equation*}
	\mathrm{Tr}_n(\phi,h;\t):=\sum_{\lambda\vdash n} \phi(\lambda)h_\lambda(\t),
\end{equation*}
where the sum ranges over all partitions of $n$ and for $\lambda=(1^{m_1}, 2^{m_2}, \ldots, n^{m_n})\vdash n$, we set
\begin{align*}
	h_\lambda(\t) := \prod_{j=1}^n h_j^{m_j}(\t).
\end{align*}
{\it Partition Eisenstein traces} are the partition traces with respect to the sequence of Eisenstein series $G=\{G_k\}_{k\in\N}$. Here, $G_{k}$ is the {\it Eisenstein series}\footnote{In \cite{Traces} the authors wrote $G_{k}$ for what here is $2G_{k}$ and $\mathrm{Tr}_k(\phi;\tau)$ for what is $\mathrm{Tr}_{2k}(\phi,G;\tau)$ here.} {\it of weight $k\in2\N$}, given by
	\begin{equation*}
		G_{k}(\tau) := -\frac{B_{k}}{2k} + \sum_{n,m\geq 1} m^{k-1} q^{nm} \qquad (q:=e^{2\pi i \tau},\, \tau\in\h:=\{w\in\mathbb C:\operatorname{Im}(w)>0\}),
	\end{equation*}
with $B_k$ the $k$-th Bernoulli number. By convention, $G_{k}:=0$ for $k$ odd. Moreover, define
	\begin{equation}\label{eq:phi-lambda}
		\phi(\lambda):=\prod_{j=1}^k \frac{2^{m_j}}{m_j!j!^{m_j}}.
	\end{equation}
The following result was obtained in {\cite[Theorem~1.2]{Traces}}.
\begin{thm}\label{thm:traces}
	We have
	\begin{align*}
		\sum_{k\geq 0} C_{k}(q) \frac{z^{k}}{k!} &= \frac{2\sinh\!\left(\frac{z}{2}\right)}{z(q)_\infty} \sum_{k\geq 0} \mathrm{Tr}_k(\phi,G;\t) z^{k}.
	\end{align*}
\end{thm}
The Eisenstein series $G_k$ play a key role in the theory of modular forms and satisfy many interesting properties. They have a beautiful connection with Bernoulli numbers in that, for $k\ge2$, we have
\begin{align*}
	\lim_{\t\to i\infty} G_k(\tau) = -\frac{B_k}{2k}.
\end{align*}
For $k>2$, $G_k$ is a modular form of weight $k$ on $\mathrm{SL}_2(\Z)$ and $G_2$ is quasimodular (see Subsection~\ref{subsec:modularandquasi} for the definition). More precisely, we have for $\left(\begin{smallmatrix}
	a&b\\ c&d
\end{smallmatrix}\right)\in\mathrm{SL}_2(\Z)$,
\begin{equation}\label{eq:Gk}
				G_k\left(\frac{a\tau+b}{c\tau+d}\right) = \begin{cases}
		(c\tau+d)^{k} G_{k}(\tau)&\text{if $k\neq2$},\\
		(c\tau+d)^2G_2(\tau) + \frac{ic}{4\pi}(c\tau+d) & \text{if $k=2$}.
	\end{cases}
		\end{equation}
			 Another key property of the Eisenstein series is that the corresponding algebra $\mathbb{Q}[G_2,G_4,G_6,\ldots]$ of quasimodular forms is closed under the action of $D:=q\frac{\partial}{\partial q}$.

\subsection{Main results} A result like Theorem~\ref{thm:traces} involving the {\it rank moments}
\begin{equation}\label{eq:rankmoments}
	R_{k}(q) := \sum_{\substack{n\geq 0\\m\in \Z}} m^{k} N(m,n) q^n
\end{equation}
may give rise to an interesting generalization of the Eisenstein series. In this paper, we show that this is indeed the case. We obtain a surprisingly simple definition of Eisenstein-type series\footnote{We emphasize that these functions are not constructed by averaging, as is done for classical Eisenstein series. They are called Eisenstein-type because they satisfy properties analogous to those of classical Eisenstein series.} $g_\ell$ in \eqref{eq:gl} which seem not to be studied in the literature before. 
As mentioned above, the rank moments are related to mock modular forms. Hence, they do not admit an expansion as a trace of modular Eisenstein series like crank moments. As our first result, we write the rank moments as traces of the Eisenstein--type series $f_k$ defined in Subsection~\ref{Modularityandcompletion}. The aim of this paper is to understand these functions $f_k$. We call a real-analytic function $f^*(\tau,\overline\tau)$ a {\it quasi-completion of $f(\tau)$}, if $f^*(\tau,\overline\tau)$ transforms like a  quasimodular form (see \eqref{eq:qmftrans2}) and if\footnote{Here and throughout, we consider $\tau$ and $\overline\tau$ as independent variables.} $\lim_{\overline\tau\to-i\infty}$ $f^*(\tau,\overline\tau)=f(\tau)$. If a quasi-completion~$f^*(\tau,\overline\tau)$ transforms as a modular form, then we call $f^*(\tau,\overline\tau)$ a \emph{completion}. If it is clear from the context, then we also just write $f^*(\tau)$ instead of $f^*(\tau,\overline\tau)$.
Our first result is the following theorem.
\begin{thm}\label{thm:main}
	There exists a family of functions $f=\{f_{k}\}_{k\in\N}$
	such that
	\begin{align}\label{eq:R_k-f_k}
		\sum_{k\ge 0} R_{k}(q) \frac{z^{k}}{k!} &= \frac{2\sinh\!\lrb{\frac z2}}{z(q)_\infty}\sum_{k\ge0} \mathrm{Tr}_k(\phi,f;\t) z^{k},
	\end{align}
	where $\phi$ is defined in \eqref{eq:phi-lambda}
	and $f_{k}$ has the following properties:
	\begin{enumerate}[leftmargin=*, label={\normalfont(\arabic*)}]
		\item\label{it:i} For $k\ge2$, we have
		\begin{align*}
			\lim_{\tau\to i\infty} f_{k}(\tau) = -\frac{B_{k}}{2k}.
		\end{align*}
		\item\label{it:ii} The function $f_{k}$ has a quasi-completion\footnote{We consider quasi-completions, because they are more analogous to the functions in Theorem~\ref{thm:traces}. In the sequel, we also study the completions $\widehat{f_k}$ of $f_k$; see \eqref{Rs} and \eqref{eq:F-hat} for their definitions.} $f^*_{k}$ which satisfies, for $\left(\begin{smallmatrix}
			a&b\\ c&d
		\end{smallmatrix}\right)\in\mathrm{SL}_2(\Z)$,
		\begin{equation*}
				f^*_{k}\left(\frac{a\tau+b}{c\tau+d}\right) = \begin{cases}
		(c\tau+d)^{k} f^*_{k}(\tau)&\text{if $k\neq2$},\\
		(c\tau+d)^2f^*_2(\tau) + \frac{3ic}{4\pi}(c\tau+d)&\text{if $k=2$}.
	\end{cases}
		\end{equation*}
		\item\label{it:iii} The algebra $\mathcal{F}:=\mathbb{Q}[f_2,f_4,\ldots, G_2, G_4,\ldots]$ is closed under the action of $D$.
	\end{enumerate}
\end{thm}
\nopagebreak
\begin{rem*}\hspace{0cm}
	\begin{enumerate}[leftmargin=*, label={\rm(\arabic*)}]
	\item As for the Eisenstein series, we have $f_k=0$ if $k$ is odd; see Remark~\ref{rem:f_odd}.
	\item The functions $f_k$ and $f_k^*$ are of intrinsic interest, as they are a step up in complexity compared to quasimodular forms.
	\item The transformation of $f_k^*$ agrees with that of $G_k$ up to a factor of $3$ in front of $c\tau+d$. This is explained by the (mock) Jacobi forms underlying the rank and crank statistics: namely, for the rank the corresponding index is $-\frac{3}{2}$ and for the crank it is $-\frac{1}{2}$.
	\end{enumerate}
\end{rem*}

Even though the $f_k$ are uniquely determined in Theorem~\ref{thm:main}, properties (1), (2), and (3) itself do not determine them uniquely. For example, after adding a cusp form of weight $k$ to each $f_k$, the resulting functions still satisfy these properties. To describe the $f_k$ uniquely without involving \eqref{eq:R_k-f_k}, we give two recursive definitions for them involving divisor-like sums.
For this, we define $g_\ell\in \mathcal{F}$ by $g_0:=1$, $g_\ell:=0$ for $\ell\in\N$ odd, and, for $\ell\in\N$ even, 
\begin{equation}\label{eq:gl}
g_{\ell}(\tau) := \left(1-2^{\ell-1}\right)\frac{ B_{\ell}}{2\ell} + \sum_{\substack{2n-1\geq 3m\geq 3}} (2n-3m)^{\ell-1}  q^{nm}-\sum_{\substack{n-1\geq 6m\geq 6}}(n-6m)^{\ell-1}  q^{nm}.
\end{equation}
For $r\in\N$, the $r$-th Fourier coefficient of $g_\ell$ is a polynomial in some of the divisors of $r$: namely those positive divisors $n,m$ satisfying the inequality $2n-1\geq3m$ or $n-1\geq6m$. Hence, the $g_k$ are in spirit of the ``mock Eisenstein series'' studied by Zagier \cite[p.15]{Zag09} and Mertens--Ono--Rolen \cite[equation~(1.4)]{MOR21}. Moreover, compare $g_\ell$ with 
\[
F_\ell^{[2]}(\tau) := \sum_{\substack{n>m\ge1 \\n-m \text{ odd }}}(-1)^n m^{\ell-1} q^{\frac{nm}{2}}
\]
which can be expressed as the Rankin--Cohen bracket of a mock modular form and $\eta^3$ as explained in \cite[Example~4 on p.~44]{DMZ}. This function goes back to \cite{EOYT, ET}, where the authors studied characters of the $\mathcal{N} = 4$ superconformal algebra in two dimensions and the elliptic genus of K3 surfaces.
\begin{thm}\label{T:Recursion}
	Let $n\in \N$.
	\begin{enumerate}[leftmargin=*, label={\normalfont(\arabic*)}]
	\item We have 
	\begin{align*}
		f_n(\tau) &= \frac n{2^{n-1}}g_n(\tau) - \sum_{\ell=2}^{n-2} \frac{\ell}{2^{\ell-2}}\binom{n-1}{\ell}f_{n-\ell}(\tau) g_{\ell}(\tau).
	\end{align*}
	\item We have
	\begin{align*}
	\	f_n(\tau) &= \sum_{\ell=2}^n \frac{(n-1)!\ell}{(\ell-1)!2^{\ell-1}} g_{\ell}(\tau) \mathrm{Tr}_{n-\ell}(\psi,f;\t),
	\end{align*}
	with $f:=\{ f_{k}(\tau)\}_{k\in\N}$ and $\psi(\lambda):= (-1)^{\sum_{j=1}^k m_j} \phi(\lambda)$.
	\end{enumerate}
\end{thm} 
As a direct corollary, we see that the algebra $\mathcal{G}:=\mathbb{Q}[g_2,g_4,\ldots, G_2, G_4,\ldots]$ is closed under~$D$, and that the elements of $\mathcal{G}$ admit a quasi-completion:
\begin{cor}
We have $\mathcal{G} = \mathcal{F}$, where the algebra $\mathcal{F}$ is given by Theorem~\ref{thm:main}\ \ref{it:iii}.
\end{cor}

The recursive formulas in Theorem $\ref{T:Recursion}$ suggest that the Fourier coefficients of $f_k$ may have large denominators. However, we show in the following theorem that, with the exception of the constant term, all Fourier coefficients of~$f_k$ are integers, similar as for~$ G_k$.

\begin{thm}\label{thm:integrality}
For $k\geq 2$, the Fourier coefficients of $f_k+\frac{B_k}{2k}$ are integers. 
\end{thm}

Finally, in part (1) of the next theorem, we give an explicit formula for $D(f_k)$ and explain how the raising and lowering operators act on the completions~$\widehat{f}_k$ of the\footnote{See Subsection~\ref{subsec:modularandquasi} for the definition; in particular, $\widehat{f}_k$ and $f_k^*$ agree if $k\neq 2$.} $f_k$. 
More precisely, let $\widehat{\mathcal{F}}:=\mathbb{C}[\widehat f_2,\widehat f_4, \widehat f_6,\ldots,\widehat{G}_2,G_4,G_6]$ and define the {\it raising} and the {\it lowering operator} by 
\begin{equation*}
	\mathcal R_k := 2i\frac\partial{\partial\tau} + \frac kv, \qquad L:=-2iv^2\frac{\partial}{\partial\overline{\t}} \qquad\qquad (k\in \Z, \tau=u+iv).
\end{equation*}
In parts (2) and (3), we obtain a recursive expression for the action of the raising and lowering operators on $\widehat{f}_k$. 
In physics literature, the latter of these falls into the realm of ``holomorphic anomaly equations''.
In string theory, mirror symmetry, enumerative geometry, and topological field theory, a \emph{holomorphic anomaly} refers to the phenomenon where a generating function (or its coefficients), instead of being purely holomorphic, acquires a controlled dependence on both $\tau$ and $\overline{\tau}$. A holomorphic anomaly equation describes this controlled dependency through a recursive relation that expresses the anti-holomorphic variation in terms of functions of lower genus or weight. While these may themselves carry non-holomorphic dependence, the anomaly equation organizes this structure in a precise and constrained way (see for example \cite{Hori2003, HosonoSaitoTakahashi1999, Manschot2019, Oberdieck2022}).
We let $\delta_{\mathcal S}:=1$ if a statement~$\mathcal S$ holds and $\delta_{\mathcal{S}}:=0$ otherwise.
\begin{thm}\label{thm:invariant-space}
	We have the following.
	\begin{enumerate}[leftmargin=*,label=\rm(\arabic*)]
		\item\label{it:Di} For $k\ge2$, we have
		\begin{equation*}
			D(f_k(\tau))  =  \frac{k!}{6} \mathrm{Tr}_{k+2}(\phi,3G-f;\tau)-\frac{k-1}{6(k+1)}f_{k+2}(\tau)	-\frac{1}{3}\sum_{a=1}^{k-1} \binom{k}{a} f_{a+1}(\tau)f_{k-a+1}(\tau).
		\end{equation*}
		\item The algebra $\widehat{\mathcal{F}}$ is closed under the raising operator. In particular, for $k\ge2$ we have
		\begin{multline*}
			\hspace{.5cm}-\frac{1}{4\pi}\mathcal{R}_k\!\left(\widehat{f_k}(\tau)\right)  =  \frac{k!}{6} \mathrm{Tr}_{k+2}\left(\phi,3\widehat G-\widehat f;\tau\right)-\frac{k-1}{6(k+1)} \widehat f_{k+2}(\tau)\\
			-\frac{1}{3}\sum_{a=1}^{k-1} \binom{k}{a}  \widehat f_{a+1}(\tau)  \widehat f_{k-a+1}(\tau),
		\end{multline*}
		where $\widehat{G}=\{\widehat G_k\}_{k\geq 1}$ is defined by $\widehat G_k=G_k$ for $k\neq 2$ and $\widehat{G}_2$  is the completion of $G_2$, given in \eqref{eq:G2*}.
		\item We have $L(\widehat{\mathcal{F}})\subseteq \widehat{\mathcal{F}} \oplus \sqrt{v} |\eta|^2 \widehat{\mathcal{F}}$. In particular, for $k\ge2$ we have
		\begin{align*}
		L\!\left(\widehat f_k(\tau)\right)&= -\frac{3}{8\pi}\delta_{k=2} + \frac{\sqrt{3}k!}{4\sqrt2\pi}\sqrt{v} |\eta(\tau)|^2  \mathrm{Tr}_{k-2}\left(\psi,\widehat{f};\tau\right),
	\end{align*}
	where $\psi$ is defined in Theorem~\ref{T:Recursion} {\rm(2)}.
		\item If $\widehat f\in \widehat{\mathcal{F}}$ is of weight $k$, then $(D +\frac{2k}{3} f_2^*)(\widehat f)$ transforms modular of weight $k+2$.
	\end{enumerate}
\end{thm}
\begin{rem*}
We expect that the algebra $\mathcal{F}$ is freely generated by the $f_k$ for $k$ even and $G_2, G_4$, and $G_6$ (see also question~\eqref{q:2} in Section~\ref{sec:?}). In particular, this would imply that the subalgebra $\mathbb{Q}[f_2,f_4,\ldots]$ of $\mathcal F$ is not closed under applying~$D$, since Theorem~\ref{thm:invariant-space}~\ref{it:Di} involves the modular Eisenstein series.
\end{rem*}

\subsection{Structure of the paper}
The paper is organized as follows. In Section~\ref{sec:prelim}, we provide certain preliminaries
on quasimodular forms and crank moments, the completion of the rank generating function, the rank-crank PDE, P\'olya cycle index polynomials, and finally, divisibility properties of multinomials. In Section~\ref{sec:mock_Eisenstein} and Section~\ref{sec:recursion}, we prove our theorems. The last two sections are devoted to examples and open questions.

\section*{Acknowledgements}
We thank Caner Nazaroglu, Kilian Rausch, and Don Zagier for helpful comments on a previous version of the manuscript. The first and the second authors were funded by the European Research Council (ERC) under the European Union’s Horizon 2020 research and innovation programme (grant agreement No. 101001179) and the first and the third authors by the SFB/TRR 191 ``Symplectic Structure in Geometry, Algebra and Dynamics'', funded by the DFG (Projektnummer 281071066 TRR 191).

\section{Preliminaries}\label{sec:prelim}

\subsection{Modular forms and quasimodular forms}\label{subsec:modularandquasi}
The Eisenstein series $G_k$ are modular forms of weight $k$ for~$\mathrm{SL}_2(\Z)$ for $k\ge 4$ even. If $k=2$, then we need to add a non-holomorphic part to make $G_2$ modular. 
To be more precise,
\begin{equation}\label{eq:G2*}
	\widehat{G}_2(\tau):=G_2(\tau)+\frac1{8\pi v}
\end{equation}
transforms like a modular form of weight $2$. The holomorphic part $G_2$ can be recovered by
\begin{equation*}
	\lim_{\overline\tau\to-i\infty}\widehat{G}_2(\tau)=G_2(\tau).
\end{equation*}
In general  $f:\mathbb H\to\mathbb C$ is an {\it almost holomorphic modular form of weight $k\in\Z$ and depth $s\in \N_0$}, if the following conditions hold:
\begin{enumerate}[leftmargin=*]
	\item We have, for $\begin{psmallmatrix}a&b\\c&d\end{psmallmatrix}\in\operatorname{SL}_2(\Z)$, 
	\begin{equation*}
		f\left(\frac{a\tau+b}{c\tau+d}\right) = (c\tau+d)^k f(\tau).
	\end{equation*}
	\item We have,\footnote{The $f_j$ here should not be confused with the Eisenstein--type series $f_j$.} for some holomorphic functions $f_j:\mathbb H\to\mathbb C$ with $f_s\neq0$,
	\begin{equation*}
		f(\tau) = \sum_{j=0}^s \frac{f_j(\tau)}{v^j}.
	\end{equation*}
	\item The function~$f$ grows at most polynomially in $\frac1v$ as $v\to0$.
\end{enumerate}
By convention, the function zero is an almost holomorphic modular form of depth~$-\infty$.

The holomorphic part $f_0$ is called a {\it quasimodular form of weight~$k$ and depth~$s$}. Note that
\begin{equation*}
	f_0(\tau) = \lim_{\overline\tau\to-i\infty} f(\tau).
\end{equation*}
More concretely, for a quasimodular form~$g$ of weight $k$ and depth $s$ there exist holomorphic functions $g_j:\mathbb H\to\mathbb C$ for $j\in\{0,\ldots,s\}$ with $g_0=g$ such that, for all $\begin{psmallmatrix}a&b\\c&d\end{psmallmatrix}\in\operatorname{SL}_2(\Z)$,
\begin{align*}
(c\tau+d)^{-k}g\!\left(\frac{a\tau+b}{c\tau+d}\right) \= \sum_{j=0}^s g_j(\tau)\Bigl(\frac{c}{c\tau+d}\Bigr)^j.
\end{align*}
Similarly, we say an analytic function $g^*(\tau,\overline{\tau})$ \emph{transforms like a quasimodular form} if there exist real-analytic functions $g_j^*(\tau,\overline{\tau})$ so that, for all $\begin{psmallmatrix}a&b\\c&d\end{psmallmatrix}\in\operatorname{SL}_2(\Z)$,
\begin{align}\label{eq:qmftrans2}(c\tau+d)^{-k}g^*\!\left(\frac{a\tau+b}{c\tau+d},\frac{a\overline{\tau}+b}{c\overline{\tau}+d}\right) \= \sum_{j=0}^s g_j(\tau,\overline{\tau})\Bigl(\frac{c}{c\tau+d}\Bigr)^j.
\end{align}
The space of all quasimodular forms is a free algebra with generators $G_2, G_4$, and $G_6$. 
It follows from Ramanujan's differential equations, \cite[equations (1), (2)]{Sko}
\begin{align}\label{eq:Ramanujandiff}
	D(G_2) &= - 2 G_2^2 + \frac56G_4,\quad
	D(G_4) = -8G_2G_4 + \frac7{10}G_6,\quad
	D(G_6) = -12G_2G_6 + \frac{400}{7}G_4^2,
\end{align}
that this algebra is closed under differentiation. We also require the {\it Serre derivative} (see, e.g., \cite[p.~48]{Zag123})
\begin{equation}\label{eq:Serrederivative}
\vartheta_k := D + 2k G_2
\end{equation}
which acts on modular forms of weight $k$, preserving the algebra of modular forms.

The {\it Dedekind eta function} $$\eta(\tau):=q^{\frac{1}{24}} \prod_{n\ge1} \left(1-q^n\right)$$ is a modular form of weight $\frac12$. 
It is not hard to deduce the following lemma.
\begin{lem}\label{lem:AnyF} For any holomorphic function~$f$ on $\mathbb H$, we have
\begin{equation*}
	\eta D\!\left(\frac{f}{\eta}\right)=G_2f+ D(f).
\end{equation*}
\end{lem}

\subsection{Crank moments}\label{sec:crank}
Set $\zeta:=e^{2\pi i z}$. By \eqref{eq:crank} and \cite[equation (7)]{Zag91}, the crank generating function can be expressed in terms of the Eisenstein series. Recall that $G_k=0$ if $k$ is odd.
\begin{lem}\label{lem:cranck}
	We have
	\begin{equation*}
		C(\zeta;q) = \frac{\sin(\pi z)}{\pi z(q)_\infty} \exp\!\left(2\sum_{k\ge2}G_{k}(\tau)\frac{(2\pi iz)^{k}}{k!}\right).
	\end{equation*}
\end{lem}
Taking $\tau \to i\infty$, we obtain the following lemma, which was also observed in \cite[Lemma~3.1]{AOS2024}.
\begin{lem}\label{lem:ExpB}
	We have
	\begin{equation*}
		\frac{\zeta^\frac12}{\zeta-1} = \frac1{2\pi iz} \exp\!\left(-\sum_{k\ge2}\frac{B_{k}}{k}\frac{(2\pi iz)^{k}}{k!}\right).
	\end{equation*}
\end{lem}

Recall that the Bernoulli numbers $B_n$ are defined as the constant terms of the Bernoulli polynomials $B_n(x)$ of degree $n$, which 
satisfy
\begin{align}\label{eq:BernoulliPolGS}
\sum _{n\ge0} B_{n}(X){\frac {t^{n}}{n!}} &= {\frac {te^{Xt}}{e^{t}-1}}, \\
\label{eq:BernoulliTranslation}
B_{n}(X+Y)&=\sum _{k=0}^{n}{n \choose k}B_{n-k}(X)Y^{k}, \\
\label{eq:BernoulliDerivative}
B_n'(X) &= n B_{n-1}(X).
\end{align}

\subsection{Mock modularity of the rank generating function}\label{S:Completion}\hspace{0cm}
The rank generating function can be written as a Lerch sum 
\begin{equation}\label{eq:RankLerch} R(\zeta;q) = \frac{1-\zeta}{(q)_\infty} \sum_{n\in\Z} \frac{(-1)^nq^\frac{n(3n+1)}2}{1-\zeta q^n}.\end{equation}
Note that this closely resembles the following representation of the crank generating function 
\begin{equation*}
	C(\z;q) = \frac{1-\zeta}{(q)_\infty} \sum_{n\in\Z} \frac{(-1)^n q^{\frac{n(n+1)}{2}}}{1-\zeta q^n}.
\end{equation*}

In Subsection 4.1 of \cite{B2018}, the first author defined (using different notation)
\begin{equation}\label{E:Rt}
	\hspace{-.2cm}R^\#(z; \tau):=
	\left(\frac{R(\zeta; q)}{\zeta^{\frac12}-\zeta^{-\frac12}}q^{-\frac{1}{24}}+\frac12 q^{-\frac16}\sum_{\pm} \pm \zeta^{\mp1} S(3z\pm\tau;3\tau)\right) e^{12\pi^2G_2(\tau)z^2},
\end{equation}
where ($\tau = u+iv,z=x+iy, u,v,x,y\in \mathbb R$)
\begin{equation*}
S(z; \tau):=\sum_{n\in\Z+\frac12}\left(\sgn(n)-E\left(\left(n+\frac{y}{v}\right)\sqrt{2v}\right)\right)
(-1)^{n-\frac12} q^{-\frac{n^2}{2}} e^{-2\pi inz},
\end{equation*}
with $E(y):=2\int_0^y e^{-\pi t^2}dt$. She showed, building on work of Zwegers \cite{Zw}, the following transformation:
\begin{lem}\label{lem:Rt-mod-transform}
	For $\begin{psmallmatrix}a&b\\c&d\end{psmallmatrix}\in\operatorname{SL}_2(\Z)$, we have
\begin{equation*}\label{E:Rt-mod-transform}
	\eta\!\left(\frac{a\tau+b}{c\tau+d}\right) R^\#\!\left(\frac{z}{c\tau+d}; \frac{a\tau+b}{c\tau+d}\right) =  (c\tau+d)\eta(\tau) R^\#(z;\tau).
\end{equation*}
\end{lem}

\subsection{The rank-crank PDE}
The rank-crank PDE of Atkin and Garvan \cite[Theorem 1.1]{AG} relates the rank and crank generating functions by a differential equation. To state it, we define the {\it heat operator}\,\footnote{We point out that we re-normalize the standard heat operator by multiplying by a factor of $\frac{1}{4\pi^2}$ (see \cite[p.~33]{EZ1985} for more details).} (of index $-\frac{3}{2}$) as
\[
H:=6q\frac\partial{\partial q}+\left(\zeta\frac\partial{\partial \zeta}\right)^{\!2}.
\]
This heat operator maps Jacobi forms of weight $\frac12$ to weight $\frac52$ and does not change the index. We state a modified version of the rank-crank PDE which is more convenient for us.

\begin{lem}\label{thm:rank-crankPDE}
	We have
	\begin{equation*}
		2 \left(\frac{\zeta^{\frac{1}{2}}(q)_\infty C(\zeta;q)}{1-\zeta}\right)^{\! 3}		=  (H+6G_2(\tau)) \left(\frac{\zeta^{\frac{1}{2}}(q)_\infty R(\zeta;q)}{1-\zeta}\right).
	\end{equation*}
\end{lem}
\begin{proof}
In \cite[Theorem 14.28]{BFOR}, the rank-crank PDE of Atkin--Garvan was formulated as
	\begin{equation*}
		2\eta^2(\tau) \left(\frac{\zeta^{\frac{1}{2}}q^{-\frac{1}{24}}C(\zeta;q)}{1-\zeta}\right)^{\! 3} = H \left(\frac{\zeta^{\frac{1}{2}}q^{-\frac{1}{24}}R(\zeta;q)}{1-\zeta}\right).
	\end{equation*}
Multiplying both sides by $\eta(\tau)=q^{\frac{1}{24}}(q)_\infty$ yields
	\begin{equation*}
		2 \left(\frac{\zeta^{\frac{1}{2}}(q)_\infty C(\zeta;q)}{1-\zeta}\right)^{\! 3}
		= \eta(\tau) H \left(\frac{1}{\eta(\tau)}\frac{\zeta^{\frac{1}{2}}(q)_\infty R(\zeta;q)}{1-\zeta}\right).
	\end{equation*}
We obtain the lemma, using the definition of $H$ and Lemma~\ref{lem:AnyF} with
\[
f(z;\tau)=\frac{\zeta^{\frac{1}{2}}(q)_\infty R(\zeta;q)}{1-\zeta}.\qedhere
\]
\end{proof}

We also require the following lemma, inspired by the observation in \cite{BZ} that the non-holomorphic part of a certain non-holomorphic Jacobi form, closely related to $R^\#$, is annihilated by the heat operator.
\begin{lem}\label{lem:H-on-non-hol-R}
	 We have
	 	\begin{align*}
			H\!\lrb{ R^{\#}(z;\t)e^{-12\pi^2G_2(\tau)z^2}} = -H\!\lrb{\frac{\zeta^{\frac{1}{2}} R(\zeta;q)q^{-\frac{1}{24}}}{1-\zeta}}.
		\end{align*}
\end{lem}
\begin{proof}
	Using \eqref{E:Rt}, it is enough to show
	\begin{equation*}
		H\!\left(q^{-\frac16} \zeta^{-1} S(3z+\tau;3\tau)\right) = 0.
	\end{equation*}
	From (4.1) of \cite{B2018}, we have
	\begin{equation}\label{E:S}
		q^{-\frac16} \zeta^{-1} S(3z+\tau;3\tau)
		= \!\sum_{n\in\Z-\frac16} \!\left(\sgn\!\left(n-\frac13\right)-E\!\left(\left(n+\frac yv\right)\sqrt{6v}\right)\!\right) (-1)^{n-\frac56} q^{-\frac{3n^2}2} \zeta^{-3n}.
	\end{equation}
	Thus the claim follows once we show that
	\begin{equation*}
		H\!\left(\!\left(\sgn\left(n-\frac13\right)-E\left(\left(n+\frac yv\right)\sqrt{6v}\right)\right)q^{-\frac{3n^2}2}\zeta^{-3n}\right) = 0.
	\end{equation*}
	Noting that $H(q^{-\frac{3n^2}2}\zeta^{-3n})=0$ and
	\begin{equation*}
		H(fg) = H(f)g + fH(g) + 2\lrb{\zeta\frac\partial{\partial\zeta} f} \lrb{\zeta\frac\partial{\partial\zeta} g},
	\end{equation*}
	it is enough to prove that
	\begin{equation*}
		0 = H\!\left(E\left(\left(n\!+\!\frac yv\right)\sqrt{6v}\right)\right) - \frac{3n}{\pi i} \frac\partial{\partial z} E\left(\left(n\!+\!\frac yv\right)\sqrt{6v}\right)=
		\frac3{8\pi^2v} [2\pi wE'(w)+E''(w)]_{w=\left(n+\frac yv\right)\sqrt{6v}}\raisebox{-5pt}{.}
	\end{equation*}
	Thus we want to show that
	\begin{equation*}
		2\pi xE'(x) + E''(x) = 0,
	\end{equation*}
	which holds since $E'(x) = 2e^{-\pi x^2}$ and $E''(x) = -4\pi xe^{-\pi x^2}$.
\end{proof}

\subsection{P\'{o}lya cycle index polynomials}
	We require a result about P\'{o}lya cycle index polynomials in the case of symmetric group $S_n$, the set of permutations of the symbols $x_1,x_2,\ldots,x_n$, see \cite{S1999} and \cite[Lemma 2.1]{Traces} for more details.
	\begin{lem}[Example 5.2.10 of \cite{S1999}]\label{lem:cycle-index}
		We have
		\begin{equation*}
			\sum_{{n\ge0}}\sum_{\lambda\vdash n}\prod_{k=1}^n\frac{x_k^{m_k}}{m_k!}w^n=\exp\left(\sum_{{k\ge1}}x_k w^k\right)
		\end{equation*}
		as a formal power series in $w$, where $\lambda=(1^{m_1}, 2^{m_2}, \ldots, n^{m_n})\vdash n$.
	\end{lem}
	
\subsection{Divisibility of multinomials}
Below, we require the following lemma regarding the divisibility of multinomial coefficients, which follows directly from B\'ezout's Lemma.
\begin{lem}\label{lem:div}
For $a_1,\ldots,a_\ell\in\N$ with $\sum_{j=1}^\ell a_j=n$, we have
\[
\frac{n}{\gcd(a_1,\ldots,a_\ell)}\,\Big|\,\binom{n}{a_1,a_2,\ldots,a_\ell}.
\]
\end{lem}

\section{Eisenstein--type series and the proof of Theorems~\ref{thm:main} and \ref{thm:invariant-space}}\label{sec:mock_Eisenstein}
\subsection{Modularity and completion}\label{Modularityandcompletion}
Motivated by Lemma~\ref{lem:cranck}, we define the Eisenstein--type series~$f_k$ in terms of the rank generating function as follows:
\begin{equation}\label{eq:f}
	R(\zeta;q) =: \frac{\sin\!\left(\pi z\right)}{\pi z(q)_\infty} \exp\!\left(2\sum_{k\geq 1} f_{k} (\tau)\frac{(2\pi iz)^{k}}{k!}\right).
\end{equation}
\begin{rem}\label{rem:f_odd} Since $R$ is invariant under $\z\mapsto\z^{-1}$, i.e., $z\mapsto-z$, we have $f_{k}=0$ if $k$ is odd.
\end{rem}
In the remainder of this section, we show that the $f_k$ satisfies all of the properties in Theorem~\ref{thm:main}.
To define their quasi-completions, we let
\begin{align}
	R^\circ(z,\overline z;\tau,\overline\tau) :=\ & 2i\sin(\pi z)q^{\frac{1}{24}} R^{\#}(z;\t)e^{-12\pi^2G_2(\tau)z^2} \nonumber \\
							=\ &R(\zeta;q) + \frac12 q^{-\frac18} \left(\zeta^\frac12-\zeta^{-\frac12}\right) \sum_\pm \pm\zeta^{\mp1} S(3z\pm\tau;3\tau). \label{eq:widehatR}
\end{align}
\begin{lem}\label{L:Jacobi}
We have
\begin{equation*}
	\lim_{\overline\tau\to-i\infty} R^\circ(z,\overline z;\tau,\overline\tau) = R(\zeta;q).
\end{equation*}
\end{lem}
\begin{proof}
By \eqref{eq:widehatR}, we have 
\begin{equation*}
	\lim_{\overline\tau\to-i\infty} R^\circ(z,\overline z;\tau,\overline\tau) = R(\zeta;q) + \frac12 q^{-\frac18} \left(\zeta^\frac12-\zeta^{-\frac12}\right) \sum_\pm \pm\zeta^{\mp1} \lim_{\overline\tau\to-i\infty} S(3z\pm\tau;3\tau).
\end{equation*}
From equation (4.1) of \cite{B2018}, we directly obtain
\begin{multline*}
	\lim_{\overline\tau\to-i\infty} q^{-\frac16} \zeta^{-1} S(3z+\tau;3\tau)\\
	= \sum_{n\in\Z-\frac16} \left(\sgn\left(n-\frac13\right)-\lim_{\overline\tau\to-i\infty}E\left(\left(n+\frac yv\right)\sqrt{6v}\right)\right) (-1)^{n-\frac56} q^{-\frac{3n^2}2} \zeta^{-3n}.
\end{multline*}
Now
\begin{equation*}
	\lim_{\overline\tau\to-i\infty}E\left(\left(n+\frac yv\right)\sqrt{6v}\right) = \lim_{v\to\infty} E\left(n\sqrt{6v}+\frac{\sqrt6y}{\sqrt v}\right) = \sgn(n).
\end{equation*}
As $\sgn(n-\frac13) = \sgn(n)$ for $n\in \Z-\frac{1}{6}$ and $S(3z-\tau;3\tau) = S(-3z+\tau;3\tau)$, the claimed statement follows.
\end{proof}

We next introduce quasi-completions $f_k^*$ of $f_{k}$.
First, we define
\begin{equation}\label{eq:defRhat}
	\frac{(q)_\infty R^\circ(z,\overline z; \tau, \overline\tau)}{2i\sin(\pi z)} =: \frac1{2\pi iz}\exp\!\left(2\sum_{k,\ell\ge0}f^*_{k,\ell}(\tau, \overline\tau) \frac{(2\pi iz)^k}{k!}\frac{(2\pi i\overline z)^\ell}{\ell!}\right).
\end{equation}
We let\footnote{Note that here we suppress the dependence on $\overline\tau$.} $R^*(z;\tau)$ be the constant term of $R^\circ(z,\overline z; \tau, \overline\tau)$ in the Taylor expansion in~$\overline{z}$ and set
\begin{equation*}
	f^*_k(\tau) := f^*_{k,0}(\tau,\overline\tau).
\end{equation*}
In terms of these, we define\footnote{Again $\mathbb F(z;\tau)$ depends on $\overline\tau$.} 
\begin{equation}\label{Rs}
	\mathbb F(z;\tau) := \frac{(q)_\infty R^*(z;\tau)}{2i\sin(\pi z)} = \frac1{2\pi iz}\exp\!\left(2\sum_{k\ge0}f^*_k(\tau) \frac{(2\pi iz)^k}{k!}\right).
\end{equation}
Note that $f^*_k=0$ for $k$ odd. We are now ready to prove Theorem~\ref{thm:main}~(2).
\begin{lem}\label{lem:completion}
The function $f_{k}^*$ is a quasi-completion of $f_k$. In particular, for $\left(\begin{smallmatrix}
a&b\\ c&d
\end{smallmatrix}\right)\in\mathrm{SL}_2(\Z)$,
		\begin{equation*}
				f^*_{k}\left(\frac{a\tau+b}{c\tau+d}\right) = \begin{cases}
		(c\tau+d)^{k} f^*_{k}(\tau)&\text{if $k\neq2$},\\
		(c\tau+d)^2f^*_2(\tau) + \frac{3ic}{4\pi}(c\tau+d)&\text{if $k=2$}.
	\end{cases}
		\end{equation*}
\end{lem}
\begin{proof}
We start by showing the transformation law of $f^*_k$. 
Combining \eqref{eq:defRhat} and \eqref{eq:widehatR} with \eqref{eq:Gk} and Lemma~\ref{lem:Rt-mod-transform}, we obtain
\begin{align*}
	&\frac{c\tau+d}{2\pi iz} \exp\!\left(2\sum_{k,\ell\ge0} f^*_{k,\ell}\left(\frac{a\tau+b}{c\tau+d}, \frac{a\overline\tau+b}{c\overline\tau+d}\right) \frac{\left(\frac{2\pi iz}{c\tau+d}\right)^k}{k!} \frac{\left(\frac{2\pi i\overline z}{c\overline\tau+d}\right)^\ell}{\ell!}\right)\hspace{-4.8cm}\\
	&\hspace{2cm}=\eta\left(\frac{a\tau+b}{c\tau+d}\right)R^\#\left(\frac{z}{c\tau+d}; \frac{a\tau+b}{c\tau+d}\right) e^{-12\pi^2G_2\left(\frac{a\tau+b}{c\tau+d}\right)\left(\frac{z}{c\tau+d}\right)^2}\\
	&\hspace{2cm}= (c\tau+d)\eta(\tau) R^\#(z;\tau) e^{-12\pi^2\left(G_2(\tau)+\frac{ic}{4\pi(c\tau+d)}\right)z^2}\\
	&\hspace{2cm}= \frac{c\tau+d}{2\pi iz}\exp\!\left(2\left(\frac{3 ic}{4\pi(c\tau+d)}\frac{(2\pi iz)^2}{2!} + \sum_{k,\ell\ge0} f^*_{k,\ell}(\tau,\overline\tau)\frac{(2\pi iz)^k}{k!}\frac{(2\pi i\overline{z})^\ell}{\ell!}\right)\right).
\end{align*}
Hence, we get
\begin{equation*}
	f^*_{k,\ell} \left(\frac{a\tau+b}{c\tau+d}, \frac{a\overline\tau+b}{c\overline\tau+d}\right) = \begin{cases}
		(c\tau+d)^k	 (c\overline\tau+d)^\ell f^*_{k,\ell}(\tau,\overline\tau)&\text{if } (k,\ell)\neq(2,0),\\
		(c\tau+d)^2f^*_{2,0}(\tau,\overline\tau) + \frac{3ic}{4\pi}(c\tau+d)&\text{if } (k,\ell)=(2,0).
	\end{cases}
\end{equation*}
So the transformation formula for $f^*_k$ follows.

We next show that
\begin{equation}\label{E:Limit}
	\lim_{\overline\tau\to-i\infty} f^*_k(\tau) = f_k(\tau).
\end{equation}
By \eqref{eq:defRhat}, Lemma~\ref{L:Jacobi}, and \eqref{eq:f} we have
\begin{align*}
	\frac1{2\pi iz} \lim_{\overline\tau\to-i\infty} \exp\!\left(2\sum_{k,\ell\ge0}f^*_{k,\ell}(\tau,\overline\tau)\frac{(2\pi iz)^k}{k!}\frac{(2\pi i\overline z)^\ell}{\ell!}\right) = \frac{(q)_\infty}{2i\sin(\pi z)} \lim_{\overline\tau\to-i\infty} R^\circ(z,\overline z;\tau,\overline\tau)\hspace{.5cm}\\
	 = \frac{(q)_\infty}{2i\sin(\pi z)} R(\z;q) = \frac{1}{2\pi i z}\exp\!\left(2\sum_{k\ge 0}f_k(\tau)\frac{(2\pi iz)^k}{k!}\right).
\end{align*}
Comparing coefficients gives \eqref{E:Limit}.
\end{proof}

We also define
\begin{align}\label{eq:F-hat}
	\widehat{\mathbb{F}}(z;\tau) := \frac{(q)_\infty R^*(z;\tau)e^{-\frac{3\pi z^2}{2v}}}{2i\sin(\pi z)}  =: \frac1{2\pi iz} \exp\!\left(2\sum_{k\ge1}\widehat f_k(\tau)\frac{(2\pi iz)^k}{k!}\right).
\end{align}
Then, by the same argument used to prove Lemma~\ref{lem:completion}, we obtain the following result.
\begin{lem}\label{lem:f^} For $k\in\mathbb N$, we have
\begin{equation*}
	\widehat f_k(\tau) =
	\begin{cases}
		f_k^*(\tau) & \text{if $k\neq 2$},\\
		f_2^*(\tau) + \frac{3}{8\pi v} & \text{if $k=2$}.
	\end{cases}
\end{equation*}
In particular
\begin{equation*}
\widehat f_k\left(\frac{a\tau+b}{c\tau+d}\right) = (c\t+d)^k\widehat f_k(\tau) .
\end{equation*}
\end{lem}

\subsection{Limiting behavior of $f_{k}$}\label{S:Asymp}
Next, we determine the behavior of $f_k(\tau)$ as $\tau\to i\infty$.
\begin{lem}\label{lem:limit} For $k\ge2$, we have
\begin{equation*}
	\lim_{\tau\to i\infty} f_k(\tau) = \lim_{\tau\to i\infty} f^*_k(\tau) = \lim_{\tau\to i\infty} \widehat f_k(\tau) = -\frac{B_k}{2k}.
\end{equation*}
\end{lem}
\begin{proof}
	Using \eqref{eq:RankLerch}, we first compute
	\begin{equation*}
		\frac{(q)_\infty R(\zeta;q)}{2i\sin(\pi z)} = -\frac{\zeta^\frac12(q)_\infty}{1-\zeta} R(\zeta;q) = -\zeta^\frac12 \sum_{n\in\Z} \frac{(-1)^nq^\frac{n(3n+1)}2}{1-\zeta q^n} \to -\frac{\zeta^\frac12}{1-\zeta}
	\end{equation*}
	as $\tau\to i\infty$.
	Thus, by \eqref{eq:f},
	\begin{equation*}
		\frac{\zeta^\frac12}{\zeta-1} = \frac1{2\pi iz} \exp\!\left(2\sum_{k\ge1}\lim_{\tau\to i\infty}f_k(\tau)\frac{(2\pi iz)^k}{k!}\right).
	\end{equation*}
	By Lemma~\ref{lem:ExpB}, 
	we obtain the claim for $f_k$.

	To prove the claim for $f^*_k$, by \eqref{eq:f} and \eqref{Rs}, we have to show that
	\begin{equation}\label{want}
		\lim_{\tau\to i\infty} \frac{(q)_\infty\left(R^*(z;\tau)-R(\zeta;q)\right)}{2i\sin(\pi z)} = 0.
	\end{equation}
	By \eqref{eq:widehatR}, we have
	\begin{equation*}
		R^*(z;\tau) - R(\zeta;q) = \frac1{2} q^{-\frac18} \left(\zeta^\frac12-\zeta^{-\frac12}\right) \sum_\pm \pm\zeta^{\mp1} S(3z\pm\tau;3\tau).
	\end{equation*}
	Next, note that
	\begin{equation*}
		S(3z\pm\tau;3\tau) = \sum_{n\in \Z+\frac12} \left(\sgn(n)-E\left(\left(n\pm\frac13+\frac {y}{v}\right)\sqrt{6v}\right)\right) (-1)^{n-\frac12} q^{-\frac{3n^2}2} e^{-2\pi in(3z\pm\tau)}.
	\end{equation*}
	Using Lemma~1.7 of \cite{Zw}, as $\tau\to i\infty$ (so $v\to\infty$),
	\begin{equation*}
		\sgn(n)-E\left(\left(n\pm\frac13+\frac y{v}\right)\sqrt{6v}\right) \ll e^{-\pi\left(n\pm\frac13+\frac{y}{v}\right)^26v}.
	\end{equation*}
	Thus $q^{-\frac{1}{8}}S(3z\pm\tau;3\tau)\to 0$ as $\tau\to i\infty$, from which we obtain \eqref{want}. The claim for $\widehat f_k$ follows from the claim for $f_k^*$ and Lemma~\ref{lem:f^}.
\end{proof}

\subsection{Differential equations}
Next, we look at the action of~$D$ on the~$f_k$ as well as the action of the raising and the lowering operator on its completions.
\begin{proof}[Proof of Theorem~\ref{thm:invariant-space}]
	\noindent(1) It is well-known, and follows by $\eqref{eq:Ramanujandiff}$, that $D(G_\ell)\in\mathcal F$. Thus it remains to show the formula for $D(f_k)$ for $k\in2\N$ which directly implies that $D(f_k)\in\mathcal F$. Using Theorem~\ref{thm:rank-crankPDE}, Lemma~\ref{lem:cranck}, and \eqref{eq:f}, we obtain
	\begin{align}
	&\frac{2}{(2\pi iz)^3} \exp\!\left(6\sum_{k\ge1}G_k(\tau)\frac{(2\pi iz)^k}{k!}\right)		\nonumber\\
	&=  \left(6q\frac\partial{\partial q}+\left(\zeta\frac\partial{\partial \zeta}\right)^2+6G_2(\tau)\right) \frac{1}{2\pi iz}\exp\!\left(2\sum_{k\ge1}f_k(\tau)\frac{(2\pi iz)^k}{k!}\right)\label{eq:C-R-diffeq}\\
		&= \left(\!\left(12\sum_{k\ge1}D(f_k(\tau))\frac{(2\pi iz)^k}{k!} + 6G_2(\tau)\right) \frac{1}{2\pi iz} + \frac{2}{(2\pi i z)^3} +  2\sum_{k\ge1}f_k(\tau)\frac{(k-2)(2\pi iz)^{k-3}}{(k-1)!} \right. \nonumber\\
		&\hspace{.2cm} \left.+2\left(-\frac{1}{(2\pi i z)^2} +  2\sum_{k\ge1}f_k(\tau)\frac{(2\pi iz)^{k-2}}{(k-1)!}\right)\sum_{\ell\ge1}f_\ell(\tau)\frac{(2\pi iz)^{\ell-1}}{(\ell-1)!}\right)\exp\!\left(2\sum_{k\ge1}f_k(\tau)\frac{(2\pi iz)^k}{k!}\right).\nonumber
	\end{align}
	Multiplying by $\frac{(2\pi iz)^3}{2} \exp(-2\sum_{k\ge1}f_k(\tau)\frac{(2\pi iz)^k}{k!})$
and collecting all terms with $D(f_k)$ on one side and the rest on the other, we have
	\begin{multline}
		6\sum_{k\ge1}D(f_k(\tau))\frac{(2\pi iz)^{k+2}}{k!} =- 3G_2(\tau)(2\pi i z)^2
		 +\exp\!\left(2\sum_{k\ge1}(3G_k(\tau)-f_k(\tau))\frac{(2\pi iz)^k}{k!}\right)-1\\
		-\sum_{k\ge1}f_k(\tau)(k-3)\frac{(2\pi iz)^k}{(k-1)!}-2\left(\sum_{k\ge1}f_k(\tau)\frac{(2\pi iz)^{k}}{(k-1)!}\right)^{2}.\label{D}
	\end{multline}
	Using Lemma~\ref{lem:cycle-index} with $w=2\pi i z$ and ${x_k}=\frac{2}{k!}(3G_k-f_k)$, we obtain
	\begin{equation*}
		\exp\!\left(2\sum_{k\ge1}(3G_k(\tau)-f_k(\tau))\frac{(2\pi iz)^k}{k!}\right) = \sum_{k\ge 0} \mathrm{Tr}_k(\phi,3G-f;\t) (2\pi i z)^k.
	\end{equation*}
	Plugging this in \eqref{D} and extracting the coefficient of ${(2\pi i z)^{k+2}}$ on both sides, we get, for $k\geq 2$,
	\begin{align*}
		 \frac{6}{k!}D(f_k(\tau)) &= \mathrm{Tr}_{k+2}(\phi,3G-f;\tau) -\frac{k-1}{(k+1)!} f_{k+2}(\tau) - 2\sum_{\substack{a,b\ge0\\a+b=k}}\frac{1}{a!b!}f_{a+1}(\tau)f_{b+1}(\tau).
	\end{align*}
	Multiplying with $\frac{k!}{6}$, yields the claim.

	\noindent(2) By \eqref{eq:widehatR} and Lemma~\ref{lem:H-on-non-hol-R}, we have
	\begin{equation*}
		H\!\left(\frac{1}{\eta(\tau)}\frac{\zeta^{\frac{1}{2}}(q)_\infty R^\circ(z,\overline z; \tau, \overline\tau)}{1-\zeta}\right) = H\!\left(\frac{1}{\eta(\tau)}\frac{\zeta^{\frac{1}{2}}(q)_\infty R(\z;q)}{1-\zeta}\right).
	\end{equation*}
	Using Lemma~\ref{lem:AnyF} and then Theorem~\ref{thm:rank-crankPDE}, we obtain
	\begin{equation*}
		(H+6G_2(\tau))\left(\frac{\zeta^{\frac{1}{2}}(q)_\infty R^\circ(z,\overline z; \tau, \overline\tau)}{1-\zeta}\right) = 2 \left(\frac{\zeta^{\frac{1}{2}}(q)_\infty C(\zeta;q)}{1-\zeta}\right)^{\! 3}.
	\end{equation*}
	Employing \eqref{eq:defRhat} for the left-hand side, we obtain
	\begin{equation*}
		2 \left(\frac{(q)_\infty C(\zeta;q)}{2i\sin(\pi z)}\right)^{\! 3} = \left(H+6G_2(\tau)\right)\left(\frac1{2\pi iz}\exp\!\left(2\sum_{k,\ell\ge0}f^*_{k,\ell}(\tau, \overline\tau) \frac{(2\pi iz)^k}{k!}\frac{(2\pi i\overline z)^\ell}{\ell!}\right)\right).
	\end{equation*}
	Hence, using Lemma~\ref{lem:cranck} and the definition of $H$, we obtain
	\begin{multline*}
		\frac{2}{(2\pi iz)^3} \exp\!\left(6\sum_{k\ge1}G_k(\tau)\frac{(2\pi iz)^k}{k!}\right) \\
		=\left(6q\frac\partial{\partial q}+\left(\zeta\frac\partial{\partial \zeta}\right)^2+6G_2(\tau)\right)\left(\frac1{2\pi iz}\exp\!\left(2\sum_{k,\ell\ge0}f^*_{k,\ell}(\tau, \overline\tau) \frac{(2\pi iz)^k}{k!}\frac{(2\pi i\overline z)^\ell}{\ell!}\right)\right).
	\end{multline*}
	Taking the constant terms from both sides with respect to $\overline{z}$ gives
	\begin{multline*}
		\frac{2}{(2\pi iz)^3} \exp\!\left(\!6\sum_{k\ge1}G_k(\tau)\frac{(2\pi iz)^k}{k!}\!\right) \\
		=\left(\!6q\frac\partial{\partial q}+\left(\zeta\frac\partial{\partial \zeta}\right)^2+6G_2(\tau)\!\right)\!\left(\frac1{2\pi iz}\exp\!\left(\!2\sum_{k\ge0}f^*_{k}(\tau) \frac{(2\pi iz)^k}{k!}\!\right)\right)\!.
	\end{multline*}
		Now the shape of the above equation is exactly the same as in~\eqref{eq:C-R-diffeq} except that $f_k$ is replaced by $f^*_k$ on the right-hand side. Hence, the same calculation as in part~(1) gives that
		\begin{equation*}
		 \frac{6}{k!}D\left(f^*_k(\tau)\right)
		 =		 \mathrm{Tr}_{k+2}\left(\phi,3G-f^*;\tau\right)-\frac{k-1}{(k+1)!}f^*_{k+2}(\tau)  -2\sum_{\substack{a,b\ge0\\a+b=k}} \frac{1}{a!b!} f^*_{a+1}(\tau) f^*_{b+1}(\tau).
		\end{equation*}
Noting that $3G-f^*=3\widehat G-\widehat f$, by Lemma~\ref{lem:f^}, we have
		\begin{equation*}
		 \frac{6}{k!}D\left(f^*_k(\tau)\right)
		 =		 \mathrm{Tr}_{k+2}\left(\phi,3\widehat G-\widehat f;\tau\right)-\frac{k-1}{(k+1)!}f^*_{k+2}(\tau)  -2\sum_{\substack{a,b\ge0\\a+b=k}} \frac{1}{a!b!} f^*_{a+1}(\tau) f^*_{b+1}(\tau).
		\end{equation*}
		From this, it is not hard to conclude the claim.

\noindent(3) Since $\widehat{\mathcal{F}}$ is generated by $\widehat{f}_k$ for $k\in\N$, $\widehat G_2, G_4$, and $G_6$, it is enough to show that $L(\widehat{f}_k), L(\widehat G_2)$, $L(G_4),L(G_6)\in \widehat{\mathcal{F}}\oplus \sqrt{v} |\eta|^2\widehat{\mathcal{F}}$ and prove the formula for $L(\widehat f_k)$.
First, we have $L(G_4)=L(G_6) = 0\in \widehat{\mathcal{F}}\oplus \sqrt{v} |\eta|^2\widehat{\mathcal{F}}$ since $G_4$ and $G_6$ are holomorphic.
Next, by \eqref{eq:G2*}, we have $L(\widehat{G}_2)=-\frac1{8\pi}\in \widehat{\mathcal{F}}\oplus \sqrt{v} |\eta|^2\widehat{\mathcal{F}}$. Finally, we look at $L(\widehat{f}_k)$. We compute, using \eqref{eq:F-hat},
	\begin{equation}\label{eq:L-on-F-hat}
		L\!\left(\widehat{\mathbb{F}}(z;\tau)\right)= 2\widehat{\mathbb{F}}(z;\tau) \sum_{k\ge0} L\!\left(\widehat f_k(\tau)\right) \frac{(2\pi iz)^k}{k!}.
	\end{equation}
	Using \eqref{eq:F-hat} again, the left-hand side can be written as
	\begin{equation*}
		L\!\left(\frac{(q)_\infty R^*(z;\tau) e^{-\frac{3\pi z^2}{2v}}}{2i\sin(\pi z)}\right) =\frac{(q)_\infty}{2i\sin(\pi z)} \lrb{L\!\left(R^*(z;\tau)\right)+R^*(z;\tau) \frac{3\pi z^2}{2}}e^{-\frac{3\pi z^2}{2v}}.
	\end{equation*}
	Using the above for the left-hand side and \eqref{eq:F-hat} for the right-hand side of \eqref{eq:L-on-F-hat} gives
	\begin{equation}\label{eq:L-on-R*}
		L\!\left(R^*(z;\tau)\right) = R^*(z;\tau)\sum_{k\ge0} \lrb{2L\!\left(\widehat f_k(\tau)\right) + \frac{3}{4\pi}\delta_{k=2} } \frac{(2\pi iz)^k}{k!}.
	\end{equation}

	We next compute, using \eqref{eq:widehatR},
	\begin{equation}\label{eq:L-on-R*_+-}
		L\!\left(R^*(z;\tau)\right) = \frac{1}{2} q^{-\frac{1}{8}} \left(\zeta^{\frac{1}{2}} - \zeta^{-\frac{1}{2}}\right) \sum_{\pm} \pm \zeta^{\mp 1} L\!\left(S(3z\pm\tau;3\tau)\right).
	\end{equation}

	We first consider the plus sign.
	Applying the lowering operator to \eqref{E:S} and using that by \cite[p. 11]{B2018}, we have
	\begin{equation*}
		L\!\left(E\left(\left(n+\frac yv\right)\sqrt{6v}\right)\right) = \sqrt{6} v^{\frac{3}{2}} \left(n-\frac yv\right) e^{-6\pi v\left(n+\frac yv\right)^2}
	\end{equation*}
	yields
	\begin{equation*}
	L\!\left(S(3z+\tau;3\tau)\right) = -\sqrt6 v^\frac32 q^{\frac16}\zeta
	e^{- \frac{6\pi y^2}{v}} \sum_{n\in\Z-\frac16} \left(n-\frac y{v}\right) (-1)^{n-\frac56} \overline q^{\frac{3n^2}2} \overline\z^{3n}.
	\end{equation*}

	Next, we turn to the minus sign. Using that $S$ is an even function and changing $z\mapsto -z$ in the above, yields
	\begin{align*}
		L\!\left(S(3z-\tau;3\tau)\right) = L\!\left(S(-3z+\tau;3\tau)\right) &= -\sqrt6 v^\frac32 q^{\frac16}\zeta^{-1} 
		e^{- \frac{6\pi y^2}{v}}\!\!\!\sum_{n\in\Z-\frac16}\!\! \left(n+\frac y{v}\right) (-1)^{n-\frac56} \overline q^{\frac{3n^2}2}\overline\z^{-3n}.
	\end{align*}
	Plugging the above two equations back into \eqref{eq:L-on-R*_+-} gives 
	\begin{align*}
		L\!\left(R^*(z;\tau)\right) &= -\sqrt{\frac32} v^\frac32
		e^{- \frac{6\pi y^2}{v}} q^{\frac{1}{24}} \left(\zeta^{\frac{1}{2}} - \zeta^{-\frac{1}{2}}\right) \sum_{\pm} \pm 
		\!\sum_{n\in\Z-\frac16} \left(n\mp\frac y{v}\right) (-1)^{n-\frac56} \overline q^{\frac{3n^2}2} \overline\z^{\pm 3n}\\
		&=-i\sqrt{6} v^\frac32
		e^{ \frac{3\pi (z-\overline z)^2}{2v}} q^{\frac{1}{24}} \sin(\pi z) \sum_\pm \pm \sum_{n\in \Z-\frac{1}{6}} \lrb{n\mp\frac{z-\overline z}{2iv}} (-1)^{n-\frac56} \overline q^{\frac{3n^2}2} \overline \zeta^{\pm3n},
	\end{align*}
	using that $y=\frac{z-\overline z}{2i}$.
	Using \eqref{eq:L-on-R*} and \eqref{eq:F-hat}, we obtain 
	\begin{multline*}
		\sum_{k\ge0}  \lrb{2L\!\left(\widehat f_k(\tau)\right) + \frac{3}{4\pi}\delta_{k=2} }\! \frac{(2\pi iz)^k}{k!} = \frac{L\!\left(R^*(z;\tau)\right)}{ R^*(z;\tau)} \\
		 = - \frac{i\sqrt{6} v^\frac32
		 	e^{ \frac{3\pi (z-\overline z)^2}{2v}} q^{\frac{1}{24}} \sin(\pi z)}{R^*(z;\t)} \sum_\pm \pm
		 \sum_{n\in \Z-\frac{1}{6}} \lrb{n\mp\frac{z-\overline z}{2iv}} (-1)^{n-\frac56} \overline q^{\frac{3n^2}2} \overline \zeta^{\pm3n}.
	\end{multline*}
	Collecting the constant terms in the Taylor expansions with respect to $\overline{z}$ on both sides  yields
	\begin{align*}
		&\sum_{k\ge0}  \!\lrb{\!2L\Bigl(\widehat f_k(\tau)\Bigr) \!+\! \frac{3}{4\pi}\delta_{k=2} \!}\! \frac{(2\pi iz)^k}{k!} 
		 \!=\! - \frac{i\sqrt{6} v^\frac32 e^{ \frac{3\pi z^2}{2v}} q^{\frac{1}{24}} \sin(\pi z)}{R^*(z;\t)} \sum_\pm \pm
		 \!\!\!\sum_{n\in \Z-\frac{1}{6}} \!\!\!\lrb{n \!\mp\! \frac{z}{2iv}}\! (-1)^{n-\frac56} \overline q^{\frac{3n^2}2}\\
		 &\hspace{2cm}= \frac{\sqrt{6v} z e^{\frac{3\pi z^2}{2v}} q^\frac1{24}\sin(\pi z)}{R^*(z;\tau)} \!\sum_{n\in\Z-\frac16}\! (-1)^{n-\frac56} \overline q^{\frac{3n^2}2}
		 = -iz\sqrt\frac{3v}2 \frac{\eta(\tau)}{\widehat{\mathbb F}(z;\tau)}
		 \!\sum_{n\in \Z-\frac{1}{6}}\!  (-1)^{n-\frac56} \overline q^{\frac{3n^2}2},
	\end{align*}
	using \eqref{eq:F-hat}. The sum on the right-hand side equals $-\eta(-\overline\t)$, so the above becomes
	\begin{equation*}
		iz \sqrt{\frac{3 v}{2}} \frac{|\eta(\t)|^2}{\widehat{\mathbb{F}}(z;\tau)} =-\pi \sqrt{6v} |\eta(\tau)|^2 z^2 \exp\!\left(-2\sum_{k\ge1}\widehat f_k(\tau)\frac{(2\pi iz)^k}{k!}\right),
	\end{equation*}
	by \eqref{eq:F-hat} again. Using Lemma~\ref{lem:cycle-index} with $w=2\pi i z$ and $x_k=-\frac{2}{k!}\widehat{f}_k(\tau)$, we conclude
	\begin{equation*}
		\sum_{k\ge0}  \!\lrb{\!2L\Bigl(\widehat f_k(\tau)\Bigr) + \frac{3}{4\pi}\delta_{k=2} \!}\! \frac{(2\pi iz)^k}{k!} = \frac{\sqrt{3v}}{2\sqrt2\pi} |\eta(\tau)|^2 (2\pi iz)^2 \sum_{n\ge 0} \mathrm{Tr}_n\bigl(\psi,\!\widehat{f};\!\tau\bigr)(2\pi i z)^n\!.
	\end{equation*}
	Comparing the coefficient of $(2\pi iz)^k$ gives the claim for $k\ge 2$.

\noindent(4) The Serre derivative~$\vartheta_k$ defined in \eqref{eq:Serrederivative} increases the weight of a modular object by $2$ (note that the proof in \cite[p.~48]{Zag123} goes through if $f$ is non-holomorphic). We have 
\begin{align*}
	D+\frac{2k}{3} f_2^*-\vartheta_k = 2k\left(\frac{1}3 f_2^*-G_2\right).
\end{align*}
By Lemma~\ref{lem:completion} and \eqref{eq:Gk} this is a non-holomorphic modular form of weight $2$. Hence $D+\frac{2k}{3} f_2^*$  maps a modular object of weight $k$ to a modular object of weight $k+2$.
\end{proof}
We are now ready to prove Theorem~\ref{thm:main}.
\begin{proof}[Proof of Theorem~\ref{thm:main}] Using \eqref{eq:f} and Lemma~\ref{lem:cycle-index} with $w=2\pi i z$ and $x_k = \frac{2}{k!}f_k$, we find
\begin{align}\label{eq:r-trace}
	R(\zeta;q) = \frac{\sin\!\left(\pi z\right)}{\pi z(q)_\infty} \sum_{k\ge0} \mathrm{Tr}_k(\phi,f;\t) (2\pi i z)^{k}.
\end{align}
From \eqref{eq:rank} and \eqref{eq:rankmoments} we deduce that
\begin{equation}\label{rmom}
R(\zeta;q) = \sum_{k\ge0 } R_{k}(q) \frac{(2\pi iz)^{k}}{k!}.
\end{equation}
Substituting \( z \mapsto \frac{z}{2\pi i}\) and applying \eqref{rmom} and $\sin(-\frac{iz}{2})=-i\sinh(\frac{z}{2})$ in \eqref{eq:r-trace}, we obtain
\begin{align*}
	\sum_{k\geq 0} R_k(q) \frac{z^k}{k!} 
	= \frac{2\sinh\!\left(\frac{z}{2}\right)}{z(q)_\infty} \sum_{k\ge0} \mathrm{Tr}_k(\phi,f;\t) z^{k}.
\end{align*}
By Lemma~\ref{lem:limit} the first property is satisfied, by Lemma~\ref{lem:completion} the second, and by Theorem~\ref{thm:invariant-space}~(1) the third.
\end{proof}

\section{Proof of Theorems~\ref{T:Recursion} and \ref{thm:integrality}}\label{sec:recursion}
The following lemma rewrites the rank moments~$R_k$ in terms of the~$g_\ell$, defined by \eqref{eq:gl}.
\begin{lem}\label{lem:R} For $k\geq 1$, we have
\begin{equation*}
R_k(q) = \frac{2^{2-k}}{(q)_\infty }\sum_{\substack{\ell=2 \\ \ell\equiv k\pmod{2}}}^k \binom{k}{\ell-1}\left(g_{\ell}(\tau) + \left(2^{\ell-1}-1\right) \frac{B_{\ell}}{2\ell}\right).
\end{equation*}
\end{lem}
\begin{proof}
Both sides of the lemma are zero for $k$ odd. Namely, since $N(m,n)=N(-m,n)$, we have that $R_k$ is zero for $k$ odd, and for $\ell$ odd, we have
\begin{align}\label{eq:B_l-g_l-odd}
	\left(2^{\ell-1}-1\right)\frac{ B_{\ell}}{2\ell} + g_{\ell}(\tau)=0. 
\end{align}
So, we may assume that $k$ is even. Recall that by \cite[equation (2.12)]{AS} we have, for $k\ge2$ even,
\begin{equation}\label{eq:formula_for_Rk}
R_k(q) = \frac{2}{(q)_\infty} \sum_{n\ge1} (-1)^{n+1} q^\frac{n(3n-1)}2(1-q^n)\sum_{m\geq 0} m^k q^{nm}.
\end{equation}
Distinguishing between $n$ even and $n$ odd yields
\begin{align*}
2^{k-1}(q)_\infty R_k(q)&\!=\! \!\!\! \sum_{n\geq 1, m\geq 0} \!\!\! (2m)^k \!\left(\!- q^{n(6n+2m-1)} \!+\!  q^{(2n-1)(3n+m-2)} \!+\! q^{n(6n+2m+1)} \!-\! q^{(2n-1)(3n+m-1)}\!\right)\\
&\!=\! \hspace{-.15cm}\sum_{m\geq 3n\geq 3}\hspace{-.15cm} (2m-6n)^k \!\left(\!- q^{n(2m-1)} +  q^{(2n-1)(m-2)} + q^{n(2m+1)} -  q^{(2n-1)(m-1)}\!\right)\!,
\end{align*}
making the change of variables $m\mapsto m-3n$.
Interchanging the role of $n$ and $m$ in the second and fourth sum the above becomes
\begin{multline*}
	-\sum_{m\geq 3n\geq 3} (2m-6n)^k q^{n(2m-1)} + \sum_{n\geq 3m\geq 3} (2n-6m)^k  q^{(n-2)(2m-1)} \\
			+ \sum_{m\geq 3n\geq 3} (2m-6n)^k q^{n(2m+1)} - \sum_{n\geq 3m\geq 3} (2n-6m)^k  q^{(n-1)(2m-1)}.
\end{multline*}
Making the change of variables $n\mapsto n+2$ in the second sum, the change of variables $m\mapsto m-1$ in the third sum, and the change of variables $n\mapsto n+1$ in the fourth sum, the above becomes
\begin{align}
&- \hspace{-0.3cm}\sum_{2m-1\geq 6n-1\geq 5}\hspace{-0.3cm} (2m-1-6n+1)^k q^{n(2m-1)}
	 + \hspace{-0.3cm}\sum_{2n+1\geq 3(2m-1)\geq 3}\hspace{-0.3cm} (2n+1-3(2m-1))^k  q^{n(2m-1)} \nonumber\\
	 &+\hspace{-0.3cm} \sum_{2m-1\geq 6n+1\geq 7}\hspace{-0.3cm} (2m-1-6n-1)^k q^{n(2m-1)}	- \hspace{-0.3cm} \sum_{2n-1\geq 3(2m-1)\geq 3}\hspace{-0.3cm} (2n-1-3(2m-1))^k  q^{n(2m-1)} \nonumber\\
			&\qquad=\sum_{\substack{m\geq 6n+1\geq 7\\\text{$m$ odd}}}\hspace{-0.1cm} \left((m-6n-1)^k - (m-6n+1)^k\right) q^{nm}\nonumber\\[-32pt]\nonumber\\
			&\hspace{5cm} + \sum_{\substack{2n-1\geq 3m\geq 3\\ \text{$m$ odd}}}\hspace{-0.1cm} \left((2n-3m+1)^k - (2n-3m-1)^k\right)  q^{nm}\nonumber\\
	&\qquad= \sum_{\substack{m\geq 6n+1\geq 7}}\hspace{-0.1cm} \left((m-6n-1)^k - (m-6n+1)^k\right)  q^{nm}\nonumber\\
			&\hspace{5cm}+ \sum_{\substack{2n-1\geq 3m\geq 3}}\hspace{-0.1cm} \left((2n-3m+1)^k - (2n-3m-1)^k\right)  q^{nm}.\label{eq:all-sums}
	\end{align}
We require the identity
\[
(x+1)^k-(x-1)^k = 2\sum_{\substack{1\leq \ell \leq k-1 \\ 2\nmid \ell}} \binom{k}{\ell} x^\ell = 2\sum_{\substack{2\leq \ell \leq k \\ 2\mid \ell}} \binom{k}{\ell-1} x^{\ell-1} .
\]
Using this and the fact that $k$ is even, we obtain that \eqref{eq:all-sums} equals
\begin{align*}
2\sum_{\substack{\ell=2 \\ 2\mid \ell}}^k \binom{k}{\ell-1} \lrb{-\sum_{\substack{m\geq 6n+1\geq 7}}(m-6n)^{\ell-1}  q^{nm}  + 	\sum_{\substack{2n-1\geq 3m\geq 3}} (2n-3m)^{\ell-1}  q^{nm}}.
\end{align*}
Now the lemma follows from the definition of $g_\ell$.
\end{proof}

The generating function of the $g_\ell$ is closely related to the rank generating function.
\begin{lem}\label{lem:mathcalF}
We have 
\[
	\frac{ \pi z(q)_\infty}{\sin\left(\pi z\right)}R(\zeta;q) \\
= 1 + 4\sum_{k\ge1}g_{k}(\tau) \frac{(\pi iz)^{k}}{(k-1)!}.
\]
\end{lem}
\begin{proof}
By setting $X=\frac{1}{2}$ and $t=-2\pi i z$ in \eqref{eq:BernoulliPolGS}
 we have
\begin{align}\label{eq:sin-exp}
\frac{\pi z}{\sin(\pi z)} &= \sum_{n\ge0} B_{n}\!\left(\frac{1}{2}\right) \frac{(2\pi i z)^{n}}{n!}.
\end{align}
Note that, by \eqref{eq:rankmoments}, $R_0(q)=\sum_{n\geq 0} \sum_{m\in \Z} N(m,n) q^{n} = \frac{1}{(q)_\infty}$ is the generating function of partitions.
Hence, using \eqref{rmom} and Lemma~\ref{lem:R}, we obtain
\begin{align}
&\frac{\pi z(q)_\infty}{\sin\left(\pi z\right)}	\left(R(\zeta;q)-\frac{1}{(q)_\infty}\right)
 \nonumber\\
&\hspace{1cm}=  \sum_{{n\ge0}} B_{n}\!\left(\frac{1}{2}\right) \frac{(2\pi i z)^{n}}{n!} \sum_{k\ge 1}
 2^{2-k} \sum_{\substack{\ell= 1\\ \ell\not\equiv k\Pmod{2}}}^k \binom{k}{\ell} \left( g_{\ell+1}(\tau) + \frac{\left(2^{\ell}-1\right) B_{\ell+1}}{2(\ell+1)} \right)
 \frac{(2\pi iz)^{k}}{k!} \nonumber\\
 &\hspace{1cm}=  \sum_{\substack{n\ge0\\ \ell\ge1}} \sum_{\substack{k\geq \ell\\ k\not\equiv \ell\Pmod{2}}} \binom{n}{k}  \binom{k}{\ell}  B_{n-k}\!\left(\frac{1}{2}\right)
 2^{2-k} \left( g_{\ell+1}(\tau) + \frac{\left(2^{\ell}-1\right) B_{\ell+1}}{2(\ell+1)}\right)
 \frac{(2\pi iz)^{n}}{n!},\nonumber
\end{align}
where we make the change of variables $n\mapsto n-k$. 
Using \eqref{eq:BernoulliTranslation}, \eqref{eq:BernoulliDerivative}, and the fact that $B_m(\frac{1}{2})=0$ and $B_m(0)=-\frac{\delta_{m=1}}{2}$ for $m$ odd, for $\ell\in\N$ odd and $n\in \N$ we find
\begin{align*}
\sum _{\substack{k=\ell\\ k\not\equiv\ell\Pmod{2}}}^n {n \choose k}\binom{k}{\ell}B_{n-k}\!\left(\frac{1}{2}\right)2^{\ell-k} = -\binom{n}{\ell} B_{n-\ell}(0) = \begin{cases} \frac{n}{2} & \text{if } n=\ell+1, \\ 0 & \text{otherwise}.\end{cases}
\end{align*}
Using the above and \eqref{eq:B_l-g_l-odd}, we find
\begin{align*}
&\frac{\pi z(q)_\infty}{\sin\left(\pi z\right)}\left(R(\zeta;q)-\frac{1}{(q)_\infty}\right) =  \sum_{n\ge1} \frac{n}{2^{n-2}} \left( \frac{\left(2^{n-1}-1\right) B_{n}}{2n} + g_{n}(\tau) \right) \frac{(2\pi iz)^{n}}{n!}.
\end{align*}
The result follows by adding $\frac{\pi z}{\sin(\pi z)}$ on both sides and using \eqref{eq:sin-exp} and the fact that $B_n(\frac12)=-(1-2^{1-n}) B_n$.
\end{proof}

Now, we are ready to prove Theorems~\ref{T:Recursion} and \ref{thm:integrality}.

\begin{proof}[Proof of Theorem~\ref{T:Recursion}]
(1) Using Lemma~\ref{lem:mathcalF} and then taking the derivative of \eqref{eq:f} with respect to $z$ yields
\begin{equation}
	4\pi i \sum_{k\geq 1} f_{k}(\tau)\frac{(2\pi iz)^{k-1}}{(k-1)!} \exp\!\left(\!2\sum_{k\geq 1} f_{k}(\tau)\frac{(2\pi iz)^{k}}{k!}\!\right)
	=2\pi i\sum_{k\ge0} \frac{kg_{k}(\tau)}{2^{k-2}} \frac{(2\pi iz)^{k-1}}{(k-1)!}. \label{eq:recf}
\end{equation}
Again using Lemma~\ref{lem:mathcalF} we obtain
\begin{equation*}
 2\sum_{k\geq 1} f_{k}(\tau) \frac{(2\pi iz)^{k-1}}{(k-1)!} + 2\sum_{n,\ell\geq 1} \binom{n+\ell-1}{n-1}f_n(\tau) \frac{\ell g_{\ell}(\tau)}{2^{\ell-2}} \frac{(2\pi iz)^{n+\ell-1}}{(n+\ell-1)!}=\sum_{k\ge0} \frac{kg_{k}(\tau)}{2^{k-2}} \frac{(2\pi iz)^{k-1}}{(k-1)!}.
\end{equation*}
By extracting the coefficients of $(2\pi i z)^{k-1}$, we obtain~(1).

\noindent(2) We send the exponential in \eqref{eq:recf} to the other side and apply Lemma~\ref{lem:cycle-index} with $w=2\pi i z$ and $x_k=-\frac{2f_k}{k!}$ to obtain
\begin{align*}
	2\sum_{k\geq 1} f_{k}(\tau)\frac{(2\pi iz)^{k-1}}{(k-1)!}
 =\sum_{n\ge0}  \frac{ng_{n}(\tau)}{2^{n-2}}
 \frac{(2\pi iz)^{n-1}}{(n-1)!}  \sum_{m\geq 0} \mathrm{Tr}_m(\psi,f;\t) {(2\pi i z)^m}.
\end{align*}
Comparing the coefficients of $(2\pi i z)^{k-1}$ gives (2).
\end{proof}

\begin{proof}[Proof of Theorem~\ref{thm:integrality}]
By \eqref{eq:f}, \eqref{rmom}, and Lemma~\ref{lem:ExpB}, we find that
\begin{align*}
	\exp\!\left(2\sum_{k\geq 1} \left(f_{k} (\tau)+\frac{B_k}{2k}\right) \frac{z^{k}}{k!}\right) &=1+(q)_\infty\sum_{k\ge1}R_{k}(q) \frac{z^{k}}{k!}.
\end{align*}
Taking the logarithm and expanding the right-hand side formally, we deduce that
\begin{equation*}
	2\sum_{k\geq 1} \left(f_{k} (\tau)+\frac{B_k}{2k}\right)\frac{z^{k}}{k!}
	= \sum_{n \geq 1}\frac{(-1)^{n+1}}{n}\sum_{k_1,\ldots,k_n\geq 1}z^{k_1+\ldots+k_n}\prod_{j=1}^n  (q)_\infty R_{k_j}(q) \frac{1}{k_j!} .
\end{equation*}
Let $\Omega$ be the linear map on $\mathbb{C}[\![z]\!]$ given by
$ \Omega(z^k) := k!z^k$ for $k\geq 0$. 
Applying~$\Omega$, we obtain
\begin{align}
\!\!2\sum_{k\geq 1}  \!\left(f_{k} (\tau)\!+\!\frac{B_k}{2k}\right)\! z^{k} 
&=  \sum_{n \geq 1} \frac{(-1)^{n+1}}{n} \mathcal{R}_n(z;q), \label{eq:fklog}
\end{align}
where we set
\[
\mathcal{R}_n(z;q):=\sum_{k_1,\ldots,k_n\geq 1}\!\!\binom{k_1+\ldots+k_n}{k_1,\ldots,k_n} z^{k_1+\ldots+k_n}\prod_{j=1}^n  (q)_\infty R_{k_j}(q).
\]

To show that the Fourier coefficients of $f_{k} (\tau)\!+\!\frac{B_k}{2k}$ are integral, it suffices to show that these on the right-hand side of \eqref{eq:fklog} are even integers.
By \eqref{eq:formula_for_Rk} the Fourier coefficients of $(q)_\infty R_k(q)$ are even. Hence, those of $\frac{1}{n}\mathcal{R}_n(z;q)$ are even. Let $k_1,\ldots,k_n\in\N$. There exists a unique partition $\lambda$ of length~$n$ associated to these~$k_j$ (obtained by ordering the~$k_j$ in non-increasing order). Conversely, for a partition $\lambda$ of length~$n$, there are $\binom{n}{r_1(\lambda),r_2(\lambda),\ldots}$ many $n$-tuples $(k_1,\ldots,k_n)$ with associated partition $\lambda$, where $r_m(\lambda)$ denotes the number of parts of size~$m$ in $\lambda$. Let $\mathcal{P}$ be the set of partitions and $|\lambda|=\sum_{j}\lambda_j$ the size of $\lambda \in \mathcal{P}$. We replace the sum over $k_1,\ldots,k_n$ by a sum over partitions. Then, we have $\binom{k_1+\ldots+k_n}{k_1,\ldots,k_n} = \binom{|\lambda|}{\lambda_1,\ldots,\lambda_n}$
and 
\[
	\mathcal{R}_n(z;q)=\sum_{\substack{\lambda \in \mathcal{P}\\\ell(\lambda)=n}}
	{\binom{|\lambda|}{\lambda_1,\ldots,\lambda_n}}
	\binom{n}{r_1(\lambda),r_2(\lambda),\ldots}z^{|\lambda|}\prod_{j=1}^n  (q)_\infty R_{\lambda_j}(q).
\]
The claim follows once we show that
 \begin{equation}\label{show}
 	n\,\Big|\, \binom{|\lambda|}{\lambda_1,\ldots,\lambda_n}\binom{n}{r_1(\lambda),r_2(\lambda),\ldots}.
 \end{equation}

Let $d$ be the greatest common divisor of $r_1(\lambda)$, $r_2(\lambda), \ldots$. Since $\sum_{m\geq 1} r_m(\lambda)=n$, we have that $d\mid n$ and we write $n=rd$.  Moreover, note that one can write the partition $(\lambda_1,\ldots,\lambda_n)$ as $(\ell_1,\ldots,\ell_1,\ell_2,\ldots,\ell_2,\ldots,\ell_r,\ldots,\ell_r)$, where each $\ell_j$ is repeated $d$ times. First, we show that 
\begin{equation}\label{eq:sdiv}
d\,\Big|\,\binom{|\lambda|}{\lambda_1,\ldots,\lambda_n}. 
\end{equation}
We factorize the multinomial coefficient as 
\begin{equation}\label{split}
\binom{|\lambda|}{\lambda_1,\ldots,\lambda_n} = \binom{|\lambda|}{\tfrac{|\lambda|}{d},\ldots, \tfrac{|\lambda|}{d}} \binom{\frac{|\lambda|}{d}}{\ell_1,\ldots,\ell_r}^{\!d},
\end{equation}
where $\tfrac{|\lambda|}{d}$ occurs $d$ times in the first multinomial coefficient on the right-hand side.
All of the multinomial coefficients are integers and by Lemma~\ref{lem:div} with $n=|\lambda|$ and $a_j=\frac{|\lambda|}{d}$ for $j\in \{1,\ldots,d\}$, we find that $d$ divides the first factor in \eqref{split} and
hence, 
\eqref{eq:sdiv} holds.

Again using Lemma~\ref{lem:div}, in this case with $a_m=r_m(\lambda)$ and $d=\gcd(a_1,\ldots,a_\ell)$, we have
$
r\,|\,\binom{n}{r_1(\lambda),r_2(\lambda),\ldots}.
$ Combining this with \eqref{eq:sdiv} gives \eqref{show}.
\end{proof}

Finally, we refine Lemma~\ref{lem:limit} by computing the first Fourier coefficients of $f_k$. In particular, the $f_k$ are naturally normalized in the sense that the second Fourier coefficient is one. 
\begin{prop}\label{prop:fknorm}
For $k\geq 2$ even we have
\[f_k(\tau) = - \frac{B_k}{2k}+q^2+\left(2^{k}-1\right)q^3+O\left(q^4\right).\]
\end{prop}
\begin{proof}
Recall that $(q)_\infty R_k(q)=0$ for $k$ odd. By \eqref{eq:formula_for_Rk}, we have for $k$ even
\[(q)_\infty R_k(q)=2q^2+2\left(2^{k}-1\right)q^3+O\left(q^4\right).\]
 Hence, we have $\prod_{j=1}^n(q)_\infty R_{k_j}(q)=O(q^4)$ if $n\geq 2$ and $k_j\in \mathbb{N}$ even. Therefore, by \eqref{eq:fklog},
\begin{align*}
2\sum_{k\geq 1}  \left(f_{k} (\tau)+\frac{B_k}{2k}\right) {z^{k}} &= 2\sum_{\substack{k\geq 2 \\ k \text { even}}}\!\left(q^2+\left(2^k-1\right)q^3\right) z^k + O\!\left(q^4\right). \qedhere
\end{align*}
\end{proof}
\begin{rem*}
Using a similar argument as in the proof of Propostion~\ref{prop:fknorm}, one can show that the $n$-th Fourier coefficient of $f_k$ is a linear combination of the $k$-th powers of $\{1,\ldots,n-1\}$. 
\end{rem*}

\section{Examples}
Here, we write down the first Fourier coefficients of the Eisenstein--type series~$f_k$. We have
\begin{align*}
 f_2(\tau) & =  -\frac{1}{24}+q^2+3 q^3+5 q^4+7 q^5+9 q^6+10 q^7+13 q^8+ O\!\left(q^9\right)\!, \\
 f_4(\tau) & =  \frac{1}{240}+q^2+15 q^3+59 q^4+139 q^5+255 q^6+406 q^7+595 q^8+O\!\left(q^9\right)\!, \\
 f_6(\tau) & =  -\frac{1}{504}+q^2+63 q^3+635 q^4+2827 q^5+8199 q^6+18550 q^7+36043 q^8+O\!\left(q^9\right)\!, \\
 f_8(\tau) & =  \frac{1}{480}+q^2+255 q^3+6179 q^4+53179 q^5+253815 q^6+844966 q^7+2234875 q^8+O\!\left(q^9\right)\!.
 \end{align*}
Some examples for Theorem~\ref{thm:invariant-space}~(1) are given by
\begin{align*}
D(f_2) &= -f_2 G_2-\frac{f_2^2}{2}-\frac{f_4}{12}+\frac{3 G_2^2}{2}+\frac{G_4}{12},\\
D(f_4) &= 6 f_2^2 G_2-18 f_2 G_2^2-f_2 G_4-f_4 G_2-\frac{2}{3}  f_2^3-\frac{7 f_4 f_2}{3}-\frac{f_6}{9}+18 G_2^3+3 G_2 G_4+\frac{G_6}{30},\\
D(f_6) &= -60 f_2^3 G_2+270 f_2^2 G_2^2+15 f_2^2 G_4-540 f_2 G_2^3+30 f_4 f_2 G_2-90 f_2 G_2 G_4-f_2 G_6\\
&\hspace{0.5cm}-45 f_4 G_2^2-f_6 G_2-\frac{5 f_4 G_4}{2}+5 f_2^4-5 f_4 f_2^2-\frac{11 f_6 f_2}{3}-\frac{25 f_4^2}{4}-\frac{f_8}{8}+405 G_2^4\\
&\hspace{0.5cm}+\frac{21855 G_4^2}{3652}+135 G_2^2 G_4+3 G_2 G_6-\frac{39 G_8}{51128}, \\
D(f_8) &= 840 f_2^4 G_2-5040 f_2^3 G_2^2-280 f_2^3 G_4+15120 f_2^2 G_2^3-840 f_4 f_2^2 G_2+2520 f_2^2 G_2 G_4\\
&\hspace{0.5cm}+28 f_2^2 G_6-22680 f_2 G_2^4+2520 f_4 f_2 G_2^2-\frac{305970}{913} f_2 G_4^2+56 f_6 f_2 G_2-7560 f_2 G_2^2 G_4\\
&\hspace{0.5cm}+140 f_4 f_2 G_4-168 f_2 G_2 G_6+\frac{39 f_2 G_8}{913}-2520 f_4 G_2^3-84 f_6 G_2^2+70 f_4^2 G_2-f_8 G_2 \\
&\hspace{0.5cm}-420 f_4 G_2 G_4-\frac{14 f_6 G_4}{3} -\frac{14 f_4 G_6}{3}-56 f_2^5+\frac{280}{3} f_4 f_2^3-\frac{28}{3} f_6 f_2^2-\frac{70}{3} f_4^2 f_2-5 f_8 f_2\\
&\hspace{0.5cm}-\frac{322 f_4 f_6}{9}-\frac{2 f_{10}}{15}+13608 G_2^5+\frac{917910}{913} G_2 G_4^2+7560 G_2^3 G_4+252 G_2^2 G_6\\
&\hspace{0.5cm}+\frac{19352886 G_4 G_6}{1983949}+\frac{36751 G_{10}}{1803590}-\frac{117 G_2 G_8}{913}.
\end{align*}

\section{Questions for future research}\label{sec:?}
 We end by raising some open questions.
\begin{enumerate}[leftmargin=*]
	\item The three properties of the~$f_k$ given in Theorem~\ref{thm:main} do not determine them uniquely. Therefore, we provide two recursive definitions of~$f_k$ in Theorem~\ref{T:Recursion}. It would be interesting to find ``nice'' properties that define an Eisenstein--type series uniquely. For example, the fact that $G_k$ is a normalized Hecke eigenform of weight $k$ which does not vanish at $i\infty$ determines it uniquely. As the coefficient of~$q$ in the Fourier expansion of the $f_k$ vanishes by Proposition~\ref{prop:fknorm}, one deduces that~$f_k$ cannot be a Hecke eigenform with respect to the usual action of the Hecke operators on $q$-series. 
	\item\label{q:2} The functions $f_2,f_4,f_6,\ldots$ together with $G_2,G_4$, and $G_6$, do not seem to satisfy any algebraic relations. This has been verified numerically up to weight 24 (and also up to mixed weight 12). Is it indeed the case that the algebra $\mathcal{F}$ is free?
	\item What variations of the $g_\ell$, in particular in the range of summation, are also mock modular? More concretely, for which $a, b\in\N$ is the function
	\[
	 \sum_{\substack{an-1\geq bm\geq b}} (an-bm)^{\ell-1}  q^{nm}-\sum_{\substack{n-1\geq abm\geq ab}}(n-abm)^{\ell-1}  q^{nm}
	\]
	an Eisenstein--type series? Note that $a=2$ and $b=3$ yields the function $g_\ell$ in this paper. Are other choices of $a$ and $b$ also of particular interest? 
	Do these functions, together with those in \cite[p.15]{Zag09} and \cite[equation~(1.4)]{MOR21}, form the first examples of a theory of (higher level)  Eisenstein--type series?
\end{enumerate}

\end{document}